\documentclass[12pt]{amsart}
\usepackage{amssymb,amsfonts,amsthm,amscd,stmaryrd,dsfont,esint,upgreek}
\usepackage[numbers]{natbib}
\usepackage[mathcal]{euscript}

\usepackage[body={6.5in,8in}, top=1.5in,left=1in]{geometry}

\theoremstyle{plain}
\newtheorem{thm}{Theorem}
\newtheorem{cor}{Corollary}
\newtheorem{lem}[cor]{Lemma}
\newtheorem{prop}[cor]{Proposition}

\theoremstyle{definition}
\newtheorem{definition}[cor]{Definition}
\newtheorem{remark}[cor]{Remark}

\numberwithin{cor}{section}
\numberwithin{equation}{section}

\newcommand{\R}{\mathbb{R}}
\newcommand{\Q}{\mathbb{Q}}
\newcommand{\N}{\mathbb{N}}
\renewcommand{\d}{n}
\newcommand{\Rd}{\mathbb R^\d}
\newcommand{\ep}{\varepsilon}

\newcommand{\barH}{\overline{H}}
\newcommand{\E}{\mathbb{E}}
\newcommand{\Prob}{\mathbb{P}}

\DeclareMathOperator*{\osc}{osc}

\DeclareMathOperator{\USC}{USC}
\DeclareMathOperator{\LSC}{LSC}
\DeclareMathOperator{\BUC}{BUC}

\DeclareMathOperator*{\esssup}{ess\,sup}
\DeclareMathOperator*{\essinf}{ess\,inf}
\DeclareMathOperator*{\argmin}{argmin}

\begin{document}

\title{Stochastic homogenization of $L^\infty$ variational problems}

\author{Scott N. Armstrong}
\address{Department of Mathematics\\ University of Wisconsin, Madison\\ 480 Lincoln Drive\\
Madison, Wisconsin 53706.}
\email{armstron@math.wisc.edu}
\author{Panagiotis E. Souganidis}
\address{Department of Mathematics\\ The University of Chicago\\ 5734 S. University Avenue
Chicago, Illinois 60637.}
\email{souganidis@math.uchicago.edu}
\date{\today}

\begin{abstract}
We present a homogenization result for $L^\infty$ variational problems in general stationary ergodic random environments. By introducing a generalized notion of distance function (a special solution of an associated eikonal equation) and demonstrating a connection to absolute minimizers of the variational problem, we obtain the homogenization result as a consequence of the fact that the latter homogenizes in random environments.
\end{abstract}
\maketitle

\section{Introduction} \label{I}

We study absolute minimizers $u^\ep$ of the $L^\infty$ variational problem
\begin{equation} \label{vpep}
\mbox{minimize} \quad \esssup_{x\in U} H\big(Dv(x),\tfrac x\ep, \omega\big) \quad \mbox{subject to} \quad v = g \ \ \mbox{on} \ \partial U.
\end{equation}
The Hamiltonian $H=H(p,y,\omega)$ is a function $H:\R^n\times\R^n\times \Omega\to\R$, with $\Omega$ a probability space. Roughly speaking, $H$ is assumed to be convex and coercive in $p$, stationary and ergodic in $(y,\omega)$ and sufficiently regular in $(p,y)$. The precise assumptions are stated in Section~\ref{P}. The domain $U \subseteq \R^n$ is taken to be bounded and smooth and $g\in C(\partial U)$ is given.

Our main result is that, under certain natural conditions, an absolute minimizer $u^\ep$ of \eqref{vpep} converges, almost surely in $\omega$ and uniformly in $\overline U$, as $\ep \to 0$, to an absolute minimizer of a deterministic limiting $L^\infty$ variational problem. The effective Hamiltonian $\barH$ is the same as the one which arises in the random homogenization (see Souganidis~\cite{S} and Rezakhanlou and Tarver~\cite{RT}) of the first-order Hamilton-Jacobi equation
\begin{equation}\label{hj1}
w^\ep_t + H(Dw^\ep,\tfrac x\ep,\omega) = 0 \quad \mbox{in} \ \Rd\times (0,\infty).
\end{equation}
Our analysis brings to light the connection between these problems and strongly utilizes the fact that \eqref{hj1} homogenizes.

\nocite{AS2,ASo,BEJ,BJW,BPG,CdP,CdPP,JWY,LS}

Due to the fact that the $L^\infty$ norm is not strictly convex, minimization problems such as \eqref{vpep}, interpreted naively, are ``not properly localized" and in particular possess very severe nonuniqueness phenomena. The notion of \emph{absolute minimizer}, which is defined in Section~\ref{P}, was introduced long ago by Aronsson~\cite{A1,A2} to rectify this situation.

The theory of absolutely minimizing functions did not fully blossom until the work of Jensen~\cite{J}, who proved that absolute minimizers of $H(p) = |p|^2$ are characterized as the viscosity solutions of the degenerate elliptic equation
\begin{equation*}
\sum_{i,j=1}^\d u_{x_ix_j} u_{x_i} u_{x_j} = 0,
\end{equation*}
called \emph{infinity Laplace equation}. Since then, viscosity solution theoretic methods have been applied to the study of absolute minimizers to great effect. For general Hamiltonians which are $C^2$ and convex, absolute minimizers are characterized by the \emph{Aronsson equation}
\begin{equation}\label{aron}
\sum_{i,j=1}^\d H_{p_i} (Du,x) H_{p_j}(Du,x) u_{x_ix_j} + \sum_{i=1}^\d H_{x_i} (Du,x) H_{p_i} (Du,x) = 0.
\end{equation}
See Gariepy, Wang, Yu~\cite{GWY} and Yu~\cite{Y} for more on the Aronsson equation. We refer to Aronsson, Crandall and Juutinen~\cite{ACJ} for an introduction to $L^\infty$ variational problems.

Using this connection our main results then imply, for a smooth, convex and coercive Hamiltonian $H$, the homogenization of the boundary value problem
\begin{equation}\label{aron1}
\left\{ \begin{aligned}
& H_{p_i} (Du^\ep,\tfrac{x}{\ep}) H_{p_j}(Du^\ep,\tfrac{x}{\ep}) u^\ep_{x_ix_j} + \tfrac{1}{\ep} H_{x_i} (Du^\ep,\tfrac{x}{\ep}) H_{p_i} (Du^\ep,\tfrac{x}{\ep}) = 0 & \mbox{in} & \ U, \\
& u^\ep=g & \mbox{on} & \  \partial U,
\end{aligned} \right.
\end{equation}
to the deterministic problem
\begin{equation*}
\left\{ \begin{aligned}
& \overline H_{p_i} (D u) \overline H_{p_j}(D  u)  u_{x_ix_j} = 0  & \mbox{in} & \ U,\\
& u=g & \mbox{on} & \  \partial U,
\end{aligned} \right.
\end{equation*}
where we have employed the summation convention to simplify notation. We remark it is very difficult in most situations to determine whether $\overline H$ is even $C^1$, and for irregular $H$ the Aronsson equation is necessary but not known to be sufficient for the absolute minimizing property (see~\cite{ACJS}). Therefore, in our main results, which assert the homogenization of absolute minimizers, we prove even more.

In recent years, certain aspects of the theory of $L^\infty$ variational problems have been greatly simplified, and in particular, their study is no longer tethered to that of~\eqref{aron}. This modern point of view was initiated by Peres, Schramm, Sheffield and Wilson~\cite{PSSW} and more fully developed in the work of Armstrong and Smart~\cite{AS1,AS2} and Armstrong, Crandall, Julin and Smart~\cite{ACJS}. We adopt this perspective in this paper, and so we refer no further to the Aronsson equation.

With the precise definitions as well as hypotheses on $H$ postponed to the next section, and the identification of the effective Hamiltonian $\barH$ to Section~\ref{amdf}, and denoting $H^\ep(p,x,\omega) : = H(p,\tfrac x\ep,\omega)$, the main result is stated as follows.

\begin{thm} \label{main}
Assume that $H:\R^n\times\R^n\times\Omega \to \R$ satisfies \eqref{erghyp}, \eqref{Hstat}, \eqref{Hconv}, \eqref{Hcoe} and \eqref{Hec}, where $\Omega$ is a probability space as described in Section~\ref{P}. There exists an convex, continuous and coercive effective Hamiltonian $\barH:\R^n\to\R$ such that, if $u^\ep = u^\ep(x,\omega)$ is an absolute subminimizer (resp., superminimizer) for $H^\ep$ in $U$, then, almost surely in $\omega$, the function $u^*(x,\omega):= \limsup_{\ep\to 0} u^\ep(x,\omega)$ (resp., $u_*(x,\omega) : = \limsup_{\ep \to 0} u^\ep(x,\omega)$) is an absolute subminimizer (resp., superminimizer) for $\barH$ in $U$.
\end{thm}

\begin{thm}\label{mainduh}
In addition to the hypotheses of Theorem~\ref{main}, assume that the set $\argmin \overline H$ has empty interior and $U$ is bounded and smooth. Fix $g\in C(\partial U)$ and suppose that $u^\ep(\cdot,\omega)\in C(\overline U)$ is an absolute minimizer for $H^\ep$ in $U$ such that $u^\ep = g$ on $\partial U$. Then $u^\ep$ converges, almost surely in $\omega$ and uniformly on $\overline U$, to the unique absolute minimizer $u$ for $\barH$ in $U$ subject to $u=g$ on $\partial U$.
\end{thm}

A necessary and sufficient condition for a convex, continuous and coercive Hamiltonian $H=H(p)$ to have a comparison principle for its absolute (sub/super)minimizers is that the set $\argmin H := \{ p \in \Rd: H(p) = \min H\}$ has empty interior.  The necessity of this condition is obvious and the sufficiency is the main result of \cite{ACJS}, which also appeared in \cite{JWY} under more regularity assumptions on $H$. This explains the appearance of this extra hypothesis in Theorem~\ref{mainduh}, since the comparison principle allows us to obtain the uniqueness of the limit and hence the convergence of the full sequence. See Section~\ref{TOW} for further discussion on this topic as well as some examples for which $\argmin \overline H$ has empty interior.

Our approach to Theorems~\ref{main} and~\ref{mainduh} is outlined in Section~\ref{O}. The idea is to exhibit a connection between the absolute minimizers of \eqref{vpep} and special solutions (the ``cone functions'' for the Hamiltonian $H$) of the eikonal equation
\begin{equation} \label{e}
H(p+Du,y,\omega) = \mu.
\end{equation}
We identify $\overline H$ as the infimum over all $\mu$ for which \eqref{e} possesses a global subsolution which is strictly sublinear at infinity. We then demonstrate a comparison principle for \eqref{e} in exterior domains, provided $\mu > \overline H$, and with weak hypotheses on the growth of the solutions at infinity. This allows us to construct \emph{distance functions} for $H$, and we obtain the main results by homogenizing these distance functions.

We remark that distance functions for spacially-dependent Hamiltonians were previously introduced, using control theory formulas, by Champion and De Pascale \cite{CdP} who also obtained comparisons with absolute minimizers. Our approach is much different, brings to light the role of $\overline H$, and applies to a much more general class of convex Hamiltonians (e.g., we do not assume  $H(0,y) \equiv 0 \leq H(p,y)$, as in \cite{CdP}). Just before this paper was accepted, we learned of many similarities between some of our results in Section~\ref{amdf} and the work of Davini and Siconolfi~\cite{DS1}, who also study distance functions for stationary ergodic Hamiltonians. Their approach is different from ours and more similar to that of~\cite{CdP}.

The homogenization of $L^\infty$ variational problems in random environments has not been considered before. In the context of periodic media, Briani, Prinari and Garroni~\cite{BPG} constructed a candidate for the effective nonlinearity through a $\Gamma$-limit, although to our knowledge the periodic homogenization of absolute minimizers was left open.

This paper is organized as follows. The notation, terminology, assumptions, the definition of absolute minimizers and some auxiliary results are described in the next section. In Section~\ref{amdf} we introduce the distance functions, explain their connection with absolute minimizers and present the effective Hamiltonian and some of its properties. The homogenization of the distance functions as well as the proofs of the main results, subject to the postponement of some key ingredients, are presented in Section~\ref{O}. An example for which Theorem~\ref{mainduh} is in force is presented in Section~\ref{TOW}. In Section~\ref{MP} we study an eikonal equation and introduce a notion of generalized distance functions for Hamiltonians with spacial dependence, which we then homogenize in Sections~\ref{EH} and~\ref{H}.

\section{Preliminaries} \label{P}

We review the notation, discuss the random environment and state the assumptions on the Hamiltonian $H$, the definition of absolute minimizers as well as  some auxiliary results.

\subsection{Notations and conventions} \label{NC}
The symbols $C$ and $c$ denote positive constants, which may vary from line to line and, unless otherwise indicated, do not depend on $\omega$. We work in the $n$-dimensional Euclidean space $\R^n$ with $n \geq 1$. The sets of rational numbers and positive integers are denoted respectively by $\Q$ and $\N$. For $y \in \R^n$, we denote the Euclidean norm of $y$ by $|y|$. Open balls are written $B(y,r): = \{ x\in \R^n : |x-y| < r\}$, and we set $B_r : = B(0,r)$. If $E \subseteq \R^n$, then the closure of $E$ is denoted $\overline E$. We write $V \ll U$ if $V\subseteq \Rd$ is open and $\overline V \subseteq U$. If $U\subseteq \R^n$ is open, then $\USC(U)$, $\LSC(U)$, $\BUC(U)$, $C^{0,1}(U)$ and $C^{0,1}_{\text loc}(U)$ are respectively the sets of upper semicontinuous, lower semicontinuous, bounded and uniformly continuous, Lipschitz continuous and locally Lipschitz  continuous functions $U\to \R$.

We emphasize that, throughout this paper and unless explicitly stated to the contrary, all differential inequalities involving functions not known to be smooth are assumed to be satisfied in the viscosity sense. Wherever we refer to ``standard viscosity solution theory" in support of a claim, the details can always be traced in standard references like the book of Barles~\cite{Ba} and the \emph{User's Guide} of Crandall, Ishii and Lions~\cite{CIL}. Finally we note that we often abbreviate the phrase \emph{almost surely in} $\omega$ by ``a.s. in $\omega$."

We also stress that, while we often state or prove results only for absolutely subminimizing functions, obvious analogues for absolutely superminimizers immediately follow. This is because the definitions easily imply that $u$ is absolutely superminimizing for $H$ if and only if $-u$ is absolutely subminimizing for $\hat H(p,y) := H(-p,y)$. Hence the corresponding results for superminimizers can be easily obtained, provided we are careful to keep track of minus signs. This also accounts for the appearance of some sign changes in, for example, Definitions~\ref{CCnaive} and~\ref{CCgen} below.

\subsection{The random environment}
We consider a probability space $(\Omega, \mathcal F, \mathds P)$, and identify an instance of the ``medium" with an element $\omega \in\Omega$. The expectation of a random variable $f$ with respect to $\mathds P$ is written $\E f$. The probability space is endowed with a group $(\tau_y)_{y\in \R^n}$ of $\mathcal F$-measurable, measure-preserving transformations $\tau_y:\Omega\to \Omega$. We say that $(\tau_y)_{y\in \R^n}$ is  \emph{ergodic} if, for every $D\subseteq \Omega$ for which $\tau_z(D) = D$ for every $z\in \R^n$, either $\Prob[D] = 0$ or $\Prob[D] = 1$. An $\mathcal F$-measurable process $f$ on $\R^n \times \Omega$ is said to be \emph{stationary} if
\begin{equation*}
f(y,\tau_z \omega) = f(y+z,\omega) \quad \mbox{for every} \ y,z\in \Rd.
\end{equation*}
If $\phi:\Omega \to S$ is a random process, then $\tilde \phi(y,\omega) : = \phi(\tau_y\omega)$ is stationary. Likewise, if $f$ is a stationary function on $\Rd \times \Omega$, then $f(y,\omega) = f(0,\tau_y\omega)$. We note that the expectation of any measurable function of a stationary function is independent of the location in space, and if we are in the ergodic setting, then all supremal-type norms of a stationary function are a.s. constant.

We require a version of the subadditive ergodic theorem. To this end let $\mathcal{I}$ and $\{\sigma_t\}_{t\geq 0}$ be respectively the class of subsets of $[0,\infty)$ which consist of finite unions of intervals of the form $[a,b)$ and a semigroup of measure-preserving transformations on $\Omega$. A \emph{continuous subadditive process} on $(\Omega, \mathcal F, \Prob)$ with respect to $\{\sigma_t\}$ is a map
\begin{equation*}
Q: \mathcal I \rightarrow L^1(\Omega,\Prob)
\end{equation*}
which is
\begin{enumerate}
\item[(i)] {\it stationary}, i.e., $Q(I)(\sigma_t\omega) = Q( t + I )(\omega)$ for each $t>0$, $I \in \mathcal I$ and a.s. in $\omega$,
\item[(ii)] {\it continuous}, i.e., there exists $C> 0$ such that, for each $I\in \mathcal I$, $\E \big| Q(I) \big| \leq C |I|$, and
\item[(iii)] {\it subadditive}, i.e., if $I_1,\ldots I_k \in \mathcal I$ are disjoint and $I=\cup_j I_j$, then
$Q(I) \leq \sum_{j=1}^k Q(I_j).$
\end{enumerate}
We refer to Akcoglu and Krengel~\cite{AK} for a proof of the following proposition.

\begin{prop} \label{SAergthm}
Suppose that $Q$ is a continuous subadditive process. Then there is a random variable $a(\omega)$ such that, as $t\to\infty$,
\begin{equation*}
\frac1t Q([0,t))(\omega) \rightarrow a(\omega) \quad \mbox{a.s. in} \ \omega.
\end{equation*}
If, in addition, $\{ \sigma_t \}_{t>0}$ is ergodic, then $a$ is constant.
\end{prop}

\subsection{The precise hypotheses}
We now state the hypotheses and assumptions in our main results. We are given a probability space $(\Omega,\mathcal F, \Prob)$ and suppose that
\begin{equation} \label{erghyp}
\tau_y : \Omega \to \Omega \quad \mbox{is an ergodic group of measure-preserving transformations}
\end{equation}

The Hamiltonian $H=H(p,y,\omega)$ is a function $H: \Rd \times\Rd \times \Omega \rightarrow \R$ which is assumed to be
{\it stationary} in $(y,\omega)$ for each $p\in\Rd$, {\it convex} with respect to $p$ for each $(y,\omega)$, and for each $\omega$, {\it coercive} in $p$ uniformly in $y$ and {\it uniformly bounded and equicontinuous} locally in $p$ and uniformly in $y$. To be more explicit, we assume that:
\begin{equation} \label{Hstat}
\text{for each $p\in\Rd$}, \ \ (y,\omega) \mapsto H(p,y,\omega) \ \ \mbox{is stationary},
\end{equation}
\begin{equation} \label{Hconv}
\text{ for every $y \in \Rd$ and $\omega \in \Omega$}, \ \ p \mapsto H(p,y,\omega) \ \ \mbox{is convex},
\end{equation}
\begin{equation} \label{Hcoe}
\lim_{|p|\to \infty} \inf_{y\in \Rd} H(p,y,\omega) = +\infty \ \ \text{for every $\omega\in \Omega$,}
\end{equation}
and for every $R > 0$,
\begin{equation} \label{Hec}
 \big\{ H(\cdot,\cdot,\omega) : \omega \in\Omega \big\} \ \ \mbox{is uniformly bounded and equicontinuous on} \  B_R \times \Rd.
\end{equation}
Notice that the conditions imposed on the Hamiltonian~$H$ are taken to hold for \emph{every} $\omega \in \Omega$, rather than merely almost surely in $\omega$. This is because we lose no generality by initially removing an event of probability zero.

\subsection{Absolute minimizers}

We recall now the notion of absolute minimizers. The motivation was explained in the introduction. Following \cite{ACJS}, we split the definition into two halves and state it for the Hamiltonian $H=H(p,x)$. 
\begin{definition}
A function $u\in C^{0,1}_{\mathrm{loc}}(U)$ is called \emph{absolutely subminimizing} in $U$ for $H$ if, for every $V\ll U$ and every $v\in C^{0,1}(V)$ such that $v \leq u$ in $V$ and $v=u$ on $\partial V$,
\begin{equation} \label{AMdefc}
\esssup_{x\in V} H(Du(x),x) \leq \esssup_{x\in V} H(Dv(x),x).
\end{equation}
Likewise, $u$ is called \emph{absolutely superminimizing} for $H$ in $U$ if \eqref{AMdefc} holds provided that $V\ll U$ and $v\in C^{0,1} (V)$ is such that  $v \geq u$ in $V$ and $v=u$ on $\partial V$. Finally, $u$ is called \emph{absolutely minimizing} if it is both absolutely subminimizing and absolutely superminimizing.
\end{definition}

\subsection{Some useful results}

We state some preliminary lemmas needed in some arguments in the sequel. Several times we will use the following well-known consequence of convexity for first-order equations (see \cite{Ba}).
\begin{lem} \label{convtrick}
(i) Suppose that $u\in\USC(U)$ is a viscosity subsolution of $H(Du,x) \leq 0$ in $U$ with $H$ coercive in $p$, i.e., satisfying \eqref{Hcoer}. Then $u\in C^{0,1}(U)$ with a Lipschitz constant depending on the rate of coercivity of $H$.
(ii) Suppose that $u \in C^{0,1}_{\mathrm{loc}}(U)$ with $H$ convex in $p$. Then $u$ satisfies the inequality $H(Du,x) \leq 0$ in $U$ in the viscosity sense if and only if it satisfies the inequality a.e. in $U$.
\end{lem}

The ergodic theorem implies that a function with a stationary, mean zero gradient is strictly sublinear at infinity. This is summarized in the following lemma, which is due to Kozlov \cite{K} (a proof can also be found in the appendix of \cite{ASo}).
\begin{lem} \label{koz}
Suppose that $w:\Rd\times\Omega\to \R$ and $\Phi = Dw$ in the sense of distributions, a.s. in $\omega$. Assume $\Phi$ is stationary, $\E \Phi(0,\cdot) = 0$, and $\Phi(0,\cdot) \in L^\alpha(\Omega)$ for some $\alpha > \d$. Then
\begin{equation}  \label{sublininftyA}
\lim_{|y|\to\infty} |y|^{-1}w(y,\omega) = 0 \quad \mbox{a.s. in} \ \omega.
\end{equation}
\end{lem}

The following very simple measure theoretic lemma is cited in the proof of Proposition~\ref{mainstep}. A proof can be found in \cite[Lemma 1]{LS} or the appendix of \cite{ASo}.

\begin{lem} \label{meastheor}
Suppose that $(X,\mathcal{G}, \mu)$ is a finite measure space, and $\{ f_\ep \}_{\ep > 0} \subseteq L^1(X,\mu)$ is a family of $L^1(X,\mu)$ functions such that $\liminf_{\ep \to 0} f_\ep \in L^1(X,\mu)$, and
\begin{equation} \label{mtass}
f_\ep \rightharpoonup \liminf_{\ep \to 0} f_\ep \quad \mbox{weakly in} \ L^1(X,\mu).
\end{equation}
Then
\begin{equation*}
f_\ep \rightarrow \liminf_{\ep \to 0} f_\ep \quad \mbox{in} \ L^1(X,\mu).
\end{equation*}
In particular, $f_\ep \rightarrow \liminf_{\ep \to 0} f_\ep$ in measure.
\end{lem}

\section{The distance functions}\label{amdf}

We study the eikonal equation and describe the relationship of its solutions to absolute minimizers, introduce the effective Hamiltonian and discuss some of its properties, and present the result about the homogenization of the distance functions. Some key intermediate results are postponed to later in the paper.

\subsection{Distance functions for $H=H(p)$.} \label{babydist}
We begin our presentation with a review of the connection between distance functions and absolute minimizers in the case that $H$ is independent of $(y,\omega)$. Most of what we say here can be found in more detail in~\cite{ACJS} or \cite{GWY}.
We assume only that
\begin{equation}\label{hj2}
H:\Rd\to\R \ \text{ is convex, continuous and coercive}.
\end{equation}

This hypothesis ensures that the sublevel set $H^{-1}(\mu):=\{ q: H(q) \leq \mu \}$ is bounded for every $\mu\in \R$. The \emph{distance functions} for $H$ (called the \emph{cone functions} in~\cite{ACJS}) are defined, for every $\mu \geq \min_{\Rd}H$, by
\begin{equation} \label{disfun-ponly}
d_\mu(y) : = \max \big\{ p\cdot y \, : \, p \in H^{-1}(\mu) \} = \max\big\{ p\cdot y\, : \, H(p) \leq \mu \big\}.
\end{equation}
Select any $p_* \in \Rd$ so that $H(p_*) = \min_{\Rd} H$. It is clear that $d_\mu(y) \geq p_* \cdot y$ with equality holding only if $\mu = H(p_*)$ or $y = 0$.

It is not difficult to check that, in the viscosity sense,
\begin{equation} \label{eik}
H(Dd_\mu) = \mu \quad \mbox{in} \ \Rd \setminus \{ 0 \}.
\end{equation}
Indeed, that $d_\mu$ is a subsolution of \eqref{eik} is obvious, even in the whole space $\Rd$, since it is the maximum of a family of global subsolutions. To see that $d_\mu$ is a supersolution, assume that a smooth function $\varphi$ touches $d_\mu$ from below at a point $x_0 \neq 0$.

It then follows from the convexity of $d_\mu$, that the plane $y\mapsto D\varphi(x_0) \cdot y$ touches $d_\mu$ from below at $x_0$ as well. If we have $H(D\varphi(x_0)) < \mu$, then the continuity of $H$ yields  $H(D\varphi(x_0)+\ep x_0) \leq \mu$ for some small enough $\ep >0$, and we derive the contradiction
\begin{equation*}
d_\mu(x_0) \geq (D\varphi(x_0)+\ep x_0) \cdot x_0 > D\varphi(x_0) \cdot x_0 = d_\mu(x_0).
\end{equation*}
Thus \eqref{eik} holds in the viscosity sense for all $\mu \geq H(p_*)$.

As we prove in more generality in Section~\ref{MP}, for every $p \in \Rd$ and $\mu > H(p)$, the eikonal equation
\begin{equation} \label{Eeq}
H(p+Du) = \mu
\end{equation}
possesses a unique solution $u$ in $\Rd\setminus \{ 0 \}$ subject to
\begin{equation*}
\liminf_{|y|\to\infty} |y|^{-1} u(y) > 0 \quad \mbox{and} \quad u(0)=0.
\end{equation*}
It follows immediately from the discussion above that this solution is given by the formula $u(y) = d_\mu(y) - p\cdot y$. In other words, the distance functions $d_\mu$ give all such solutions of \eqref{Eeq} for $\mu > H(p)$.

What is more interesting (and useful) is that distance functions actually characterize absolute (sub/super)minimizers of $H=H(p)$. To see this, it is necessary to introduce the notion of {\it comparison with distance functions}.

\begin{definition} \label{CCnaive}
A bounded  $u: U\to \R$  satisfies comparisons with distance functions from above (with respect to $H$ in $U$), if
\begin{equation*}
\max_{x\in \overline V} \big(u(x) - d_\mu(x-x_0) \big) = \max_{x\in \partial V} \big( u(x) - d_\mu(x-x_0) \big)
\end{equation*}
provided that
\begin{equation}\label{cdfc}
\mu > \min H, \ \ V \ll U \quad \mbox{and} \quad x_0 \in \Rd \setminus V.
\end{equation}
Likewise, $u$ satisfies \emph{comparisons with distance functions from below} if \eqref{cdfc} implies that
\begin{equation*}
\min_{x\in \overline V} \big(u(x) + d_\mu(x_0-x) \big) = \min_{x\in \partial V} \big( u(x) + d_\mu(x_0-x) \big).
\end{equation*}
\end{definition}

The connection between between absolute minimizers and distance functions is summarized in the following result.

\begin{prop}[{\cite[Theorem 4.8]{ACJS}}] \label{CC}
Suppose that $u:U\to \R$ is bounded. Then $u$ is an absolute subminizer (superminimizer) for $H$ in $U$ if and only if $u$ satisfies comparisons with distance functions from above (below) with respect to $H$ in $U$.
\end{prop}

The hypotheses of~\cite{ACJS} include that the level sets of $\barH$ have empty interior. However, as pointed out in the introduction of \cite{ACJS}, this assumption is needed only in the proof of \cite[Lemma~5.1]{ACJS}, which is independent of \cite[Theorem 4.8]{ACJS}.

The observation behind Proposition~\ref{CC} goes back to Evans, Crandall and Gariepy~\cite{CEG} who discovered it in the case $H(p)=|p|^2$, i.e., in the context of infinity subharmonic functions. It was subsequently generalized to $H=H(p) \in C^2$ in~\cite{GWY}, the regularity assumption being finally removed in \cite{ACJS}. Since we need to apply it to $\barH$, the regularity of which is we know nothing about, this generality is essential to our approach.

\subsection{Distance functions for $H=H(p,y)$}
\label{bigdist}
Building on a connection to global subsolutions of the eikonal equation discovered in \cite{ASo}, we define distance functions as the unique solutions of the eikonal equation with specified growth at infinity. We consider $H\in C(\Rd\times\Rd)$ that is convex in $p$, that is, for each $x\in \Rd$,
\begin{equation} \label{Hconvex}
p\mapsto H(p,x) \quad \mbox{is convex,}
\end{equation}
and coercive, i.e.,
\begin{equation} \label{Hcoer}
\lim_{|p| \to \infty} \inf_{x\in \Rd} H(p,x) = +\infty,
\end{equation}
and regular in the sense that, for each $R> 0$,
\begin{equation} \label{Huc}
H \quad \mbox{is uniformly continuous and bounded on} \ \ B_R \times \Rd.
\end{equation}

Notice that a constant function is a global subsolution of the eikonal equation
\begin{equation} \label{EeqX}
H(p+Du,y) = \mu
\end{equation}
for $\mu = \sup_{y\in \Rd} H(p,y)$, and this quantity is finite by \eqref{Huc}. Therefore, we may define
\begin{multline} \label{ssass}
\overline H(p): = \inf \Big\{ \mu \in \R \, : \, \mbox{there exists a global subsolution} \ \ w\in C^{0, 1}(\Rd) \ \ \mbox{of} \ \eqref{EeqX} \\
\mbox{satisfying} \ \ \lim_{|y| \to \infty} |y|^{-1} w(y) = 0 \Big\}.
\end{multline}
It is clear from Lemma~\ref{convtrick} that
\begin{equation} \label{infsup}
\overline H(p)=\inf_{\phi\in {\mathcal S}}\esssup_{y\in\Rd}{H(p+D\phi(y),y)},
\end{equation}
where we define
\begin{equation*}
{\mathcal S}:=\{w \in C^{0,1}(\Rd): \lim_{|y| \to \infty} |y|^{-1} w(y) = 0 \}.
\end{equation*}
In some of the arguments below it is helpful to keep in mind that, in view of Lemma~\ref{convtrick}, the notion of subsolution in \eqref{ssass} may be interpreted either in the viscosity or the almost everywhere senses as the two are equivalent.

The effective nonlinearity $\overline H$ inherits the properties of convexity, coercivity and continuity from $H$, as we show in the next lemma. In particular, $\overline H$ possesses its own set of distance functions, as defined in the previous subsection, which we denote by $\overline d_\mu$.

\begin{lem} \label{Hbarprop}
$\overline H$ is convex, continuous and coercive.
\end{lem}
\begin{proof}
Any function which is strictly sublinear at infinity is touched from above at some point of $\Rd$ by the function $\ep (1+|y|^2)^{1/2}$, for any $\ep > 0$. From this and \eqref{Huc} it follows that $\overline H(p) \geq \inf_{y\in \Rd} H(p,y)$. Therefore $\overline H(p)$ is finite and, by \eqref{Hcoer}, $\overline H(p) \rightarrow +\infty$ as $|p| \to \infty$. The continuity of $\overline H$ is easy to obtain from \eqref{Huc}, and the convexity of $\overline H$ from \eqref{Hconvex}.
\end{proof}

Using ideas from \cite{ASo} we show that, for every $p\in \Rd$, $\mu > \overline H(p)$ and $x_0\in \Rd$, the eikonal equation \eqref{EeqX} possesses a unique solution in the punctured space $\Rd\setminus \{ x_0 \}$ up to the addition of constants and subject to a one-sided growth condition at infinity.

\begin{prop} \label{metexistence}
Assume that $H=H(p,y)$ satisfies \eqref{Hconvex}, \eqref{Hcoer} and \eqref{Huc}. Then for each $p\in \Rd$, $\mu > \overline H(p)$ and $x_0 \in \Rd$, there exists a unique solution $d_{x_0,\mu,p}$ of \eqref{EeqX} in $\Rd \setminus \{ x_0 \}$ satisfying
\begin{equation} \label{EeqBC}
\liminf_{|y|\to \infty} |y|^{-1} d_{x_0,\mu,p} (y) > 0 \quad \mbox{and} \quad d_{x_0,\mu,p}(x_0) = 0.
\end{equation}
\end{prop}

It follows (see Remark~\ref{redud}) that the distance functions $d_{\mu,x_0,p}$ and $d_{\mu,x_0,q}$ are redundant for $\mu > \max\{ \overline H(p),\overline H(q)\}$ in the sense that
\begin{equation*}
d_{\mu,x_0,p}(y) = d_{\mu,x_0,q}(y) + (q-p) \cdot (y-x_0).
\end{equation*}
Owing to Lemma~\ref{Hbarprop}, we can select any $p_* \in \Rd$ for which $\overline H(p_*) = \min_{\Rd} \overline H$. It follows that all the functions $\{ d_{\mu,x_0,{p}} \, : \, p \in \Rd\}$ can be described in terms of $d_{\mu,x_0,{p_*}}$. Set
\begin{equation} \label{distfundef}
d_{\mu,{x_0}} (y) : = d_{\mu,{x_0},{p_*}}(y) + p_*\cdot (y-x_0),
\end{equation}
and notice that $d_{\mu,{x_0}}$ is unambiguously defined even if $p_* $ is not the unique point at which the minimum of $\overline H$ is attained. It is clear that
\begin{equation} \label{DFoneside}
\liminf_{|y| \to \infty} |y|^{-1} \big( d_{\mu,x_0} (y) - p_*\cdot (y-x_0) \big) > 0
\end{equation}
and
\begin{equation} \label{distfuneqn}
H(Dd_{\mu,x_0} , y) = \mu \quad \mbox{in} \ \Rd\setminus \{ x_0 \}.
\end{equation}
In particular, \eqref{distfuneqn} implies, with the help of \eqref{Huc} and Lemma~\ref{convtrick}, that $d_{\mu,x_0}$ is globally Lipschitz and
\begin{equation} \label{distfunlip}
\| Dd_{\mu,x_0} \|_{L^\infty(\Rd)} \leq C_\mu.
\end{equation}
We call $\{ d_{\mu,x_0} \, : \, \mu > \min \overline H, \ x_0\in \Rd \}$ the \emph{set of distance functions for} $H$. We may now generalize the concept of comparisons with distance functions in the obvious way.

\begin{definition} \label{CCgen}
Suppose that $u: U\to \R$ is bounded. Then $u$ satisfies \emph{comparisons with distance functions from above} (with respect to $H$ in $U$) if
\begin{equation*}
\max_{x\in \overline V} \big(u(x) - d_{\mu,x_0}(x) \big) = \max_{x\in \partial V} \big( u(x) - d_{\mu,x_0}(x) \big)
\end{equation*}
provided that
\begin{equation} \label{cdfcg}
\mu > \min \overline H, \ \ V \ll U \quad \mbox{and} \quad x_0 \in \Rd \setminus V.
\end{equation}
Likewise, $u$ satisfies \emph{comparisons with distance functions from below} if \eqref{cdfcg} implies
\begin{equation*}
\min_{x\in \overline V} \big(u(x) + d_{\mu,-x_0}(-x) \big) = \min_{x\in \partial V} \big( u(x) + d_{\mu,-x_0}(-x) \big).
\end{equation*}
\end{definition}

Our definitions here agree with the ones in previous subsection. Indeed, if $H$ does not depend on $y$, then $\overline H = H$ and $d_{\mu,x_0}(y) = d_{\mu} (y - x_0)$, the latter functions being the ones appearing above and defined by \eqref{disfun-ponly}. This is clear from the uniqueness assertion in Proposition~\ref{metexistence}.

In Section~\ref{MP} we prove Proposition~\ref{metexistence} as well as the following generalization of half of Proposition~\ref{CC} (since we do not need the other half, we omit it). 

\begin{prop} \label{CCprop}
Assume \eqref{Hconvex}, \eqref{Hcoer} and \eqref{Huc}. If $u \in \USC(U)$ is an absolute subminimizer in $U$, then $u$ satisfies comparisons with distance functions from above in $U$.
\end{prop}

\section{The proof of homogenization} \label{O}

In this section we prove that the distance functions homogenize and we give the proof of Theorem~\ref{main} subject to the verification of some key intermediate results which are postponed. 

\subsection{The homogenization of the distance functions}
Much of the heavy lifting in the proof of Theorem~\ref{main} lies in the homogenization of the distance functions, which we now describe. In this subsection, the Hamiltonian $H=H(p,y,\omega)$ satisfies the hypotheses described in Section~\ref{P} and, in particular, \eqref{erghyp}, \eqref{Hstat}, \eqref{Hconv}, \eqref{Hcoe} and \eqref{Hec} are in force.

For each fixed $\omega \in \Omega$, we denote by $d_{\mu,x_0} = d_{\mu,x_0}(\cdot,\omega)$ the distance functions for the Hamiltonian $H(\cdot,\cdot,\omega)$, which are well-defined for each $p\in \Rd$, $\mu > \overline H(p,\omega)$ and $x_0\in \Rd$. Here the quantity $\overline H(p,\omega)$ is defined as in \eqref{ssass}, with respect to the Hamiltonian $H(\cdot,\cdot,\omega)$.

It is clear from the inf-sup formula \eqref{infsup} that $\overline H(p,\omega)$ is measurable in $\omega$.  The stationarity hypothesis \eqref{Hstat} yields that $\overline H(p,\omega) = \overline H(p,\tau_y\omega)$ for every $y\in \Rd$, and hence the ergodic hypothesis \eqref{erghyp} implies that $\overline H$ is constant in $\omega$; that is, there exists $\overline H(p)$ such that $\overline H (p,\omega) = \overline H(p)$ a.s.\!\! in $\omega$. Moreover, since $\overline H=\overline H(\cdot,\omega)\in C(\Rd)$ for each $\omega\in \Omega$, there exists a subset $\Omega_1\subseteq \Omega$ of full probability such that
\begin{equation} \label{Hbarconst}
\overline H(p,\omega) = \overline H(p) \quad \mbox{for every} \ \  p\in \Rd, \ \omega \in \Omega_1.
\end{equation}
Indeed, for each rational $p$ we can find a subset of full probability on which $H(p,\omega)$ is constant, and we construct $\Omega_1$ by taking the intersection of these. Since $\overline H(\cdot,\omega)$ is continuous by Lemma~\ref{Hbarprop}, we obtain \eqref{Hbarconst}.

The distance functions $d_{\mu,x_0}$ are thus well-defined for each $p\in \Rd$, $\mu > \overline H(p)$ and $x_0\in \Rd$. We consider them to be functions of $(y,\omega) \in \Rd \times \Omega$. While we do not give the argument here in order to avoid an overly pedantic presentation, we remark that $d_{\mu,x_0}$ is measurable in $(y,\omega)$, a fact which follows more or less from the uniqueness of the distance functions asserted in Proposition~\ref{metexistence}. The distance functions are stationary in the sense that
\begin{equation} \label{distfunstat}
d_{\mu,x_0} (y,\tau_z\omega) = d_{\mu,x_0+z}(y+z,\omega),
\end{equation}
a fact which is immediate from \eqref{Hstat} and uniqueness.

It follows from \eqref{Hbarconst} that the effective Hamiltonian $\overline H$ satisfies the conclusion of Lemma~\ref{Hbarprop}. In particular, $\overline H$ possesses distance functions $\overline d_{\mu}$ as described in Section~\ref{babydist}. 

The distance functions for $H^\ep(p,y,\omega):= H(p,y/\ep,\omega)$, which we denote by $d^\ep_{\mu,x_0}$, are expressed in terms of $d_{\mu,x_0}$ by
\begin{equation} \label{laidout}
d^\ep_{\mu,x_0}(y,\omega) : = \ep d_{\mu,x_0/\ep} \big( y/\ep,\omega \big).
\end{equation}
We then have the following homogenization result for the distance functions, which asserts that, as $\ep \to 0$, the distance function $d^\ep_{\mu,x_0}(y,\omega)$ converges to $\overline d_{\mu}(y-x_0)$ on a set of full probability.

\begin{prop}\label{conehmg}
There exists a subset $\Omega_0 \subseteq \Omega$ of full probability such that, for every $\mu > \min\overline H$, $x_0 \in \Rd$ and $\omega \in \Omega_0$,
\begin{equation} \label{limhmgdf}
d^\ep_{\mu,x_0}(\cdot,\omega) \rightarrow \overline d_\mu(\cdot-x_0) \quad \mbox{locally uniformly in}  \ \Rd \ \ \mbox{as} \ \ \ep \rightarrow 0.
\end{equation}
\end{prop}

The proof of Proposition~\ref{conehmg} is taken up in Sections~\ref{EH} and~\ref{H}, and is based on ideas recently developed by the authors in~\cite{ASo} as well as the earlier work of Lions and Souganidis~\cite{LS}.

\subsection{The proof of the homogenization results}
We now present the proofs of our main results, subject to the completion of the proofs of Propositions~\ref{metexistence}, \ref{CCprop}, and \ref{conehmg}.
The main idea is to exploit the connection between absolute minimizers and distance functions, thereby essentially reducing the work to that of homogenizing a first-order eikonal equation.

\begin{proof}[{Proof of Theorem~\ref{main}}]
We need only prove the first statement, since the second one follows from the first and the insertion of negative signs in appropriate places (see the comments in Section~\ref{NC}). According to Proposition~\ref{CC}, it is equivalent to argue that, a.s.~in~$\omega$, $u^*(\cdot,\omega)$ satisfies comparisons with distance functions from above with respect to $\barH$ in $U$. We denote the distance functions for $\barH$ by $\overline d_\mu$, and we may assume with no loss of generality that $\overline H(0) = \min \overline H$. Arguing by contradiction, we suppose on the contrary that there exists $\mu > \min_{p\in \Rd} \barH(p)$, $V\ll U$ and $x_0 \in \Rd\setminus V$ such that
\begin{equation} \label{bad}
\max_{x\in \overline V} \big(u^*(x,\omega) - \overline{d}_\mu(x-x_0) \big) > \max_{x\in \partial V} \big( u^*(x,\omega) - \overline{d}_\mu(x-x_0) \big).
\end{equation}
According to Proposition~\ref{conehmg}, it follows that for each $\omega$ in a subset of $\Omega$ of full probability,
\begin{equation} \label{reallybad}
\max_{x\in \overline V} \big(u^\ep(x,\omega) - d^\ep_{\mu,x_0}(x,\omega) \big) > \max_{x\in \partial V} \big( u^\ep(x,\omega) - d^\ep_{\mu,x_0}(x,\omega) \big)
\end{equation}
for small enough $\ep >0$. From this we derive a contradiction, thanks to Proposition~\ref{CCprop} and the assumption that $u^\ep$ is an absolute subminimizer for $H^\ep(\cdot,\cdot,\omega)$ in $U$.
\end{proof}

\begin{proof}[{Proof of Theorem~\ref{mainduh}}]
If $\argmin\overline H$ has empty interior, then there is a comparison principle for absolute minimizers of $H$ in bounded domains. This is the main result of \cite{ACJS}. Therefore, using Theorem~\ref{main}, we have $u^*(x,\omega) \leq u_*(x,\widetilde \omega)$ for any $\omega,\widetilde \omega\in \Omega$. We deduce that $u^*(x,\omega)=u_*(x,\omega)=:u(x)$, and so $u$ is the (necessarily unique) absolute minimizer for $\overline H$ in $U$.
\end{proof}

\subsection{An idea for an alternative proof of the homogenization results}
Given a nice (bounded) function $u:\Rd \to\R$, define the flow $T^tu(x):= v(x,t)$, where $v$ is the viscosity solution of the initial value problem
\begin{equation} \label{HJflow}
\left\{ \begin{aligned}
& v_t - H(Dv) = 0 & \mbox{in} & \ \Rd \times(0,\infty), \\
& v = u & \mbox{on} & \ \Rd\times\{ 0 \}.
\end{aligned} \right.
\end{equation}
Barron, Evans and Jensen~\cite{BEJ} conjectured (and provided a formal argument suggesting) that subsolutions $u$ of the Aronsson equation (in our language, absolute subminimizers) for a Hamiltonian $H=H(p)$ should be characterized by the property that, for every $x$,
\begin{equation*}
t\mapsto T^t u(x) \quad \mbox{is convex.}
\end{equation*}
This convexity criterion was proved for smooth $H$ by Juutinen and Saksman~\cite{JS} and for general convex $H=H(p)$ in \cite{ACJS}. For bounded domains, it is a little awkward to state the convexity criterion  in terms of the Hamilton-Jacobi flow, and for this reason \eqref{HJflow} was abandoned in \cite{ACJS} and replaced by the Hopf-Lax formula
\begin{equation} \label{HL}
T^tu(x): = \sup_{y\in U} \Big( u(y) - t L\Big(\frac{y-x}{t} \Big) \Big).
\end{equation}

At first glance, it may seem that the convexity criterion provides a more natural connection between absolute minimizers and the corresponding Hamilton-Jacobi equation. Indeed, if the convexity criterion could be generalized in an appropriate way to Hamiltonians with spacial dependence, then our Theorems~\ref{main} and~\ref{mainduh} would follow immediately from the homogenization of Hamilton-Jacobi equations~\cite{S}.

Unfortunately, it is an open problem whether the convexity criterion can be generalized to absolute minimizers of Hamiltonians with spacial dependence. The obstacle in the argument lies in showing that the absolutely subminimizing property is preserved under the flow $T^t$. This is obvious from the Hopf-Lax formula \eqref{HL} in the case $H=H(p)$, and can be shown to hold if everything is smooth, but sticky regularity issues have thus far thwarted efforts at making this rigorous.

\section{A sufficient condition for Theorem~\ref{mainduh}} \label{TOW}

The effective Hamiltonian $\overline H$ is difficult to study, even in periodic environments, and so it is not easy to determine in which situations we can expect $\overline H$ to have a ``flat spot" at its minimum, i.e., whether $\argmin \overline H$ has nonempty interior. In the periodic case and in dimension $n=1$, a Hamiltonian of the form $H(p,y) = |p|^2 + V(y)$ give rise to effective Hamiltonian $\overline H$ which can be computed explicitly (see \cite{LPV}). In this case, $\overline H$ indeed possesses a flat spot. In dimensions $n\geq 2$ the analogous situation is much more complicated, but some sufficient conditions ensuring flat spots can be found in Concordel~\cite{C}.

There is probably a connection between the appearance of flat spots for $\overline H$ and the failure of the comparison principle to hold for absolute minimizers of the corresponding Hamiltonian $H=H(p,y,\omega)$. We hope that future research will shed some light on this question. Examples found in Yu~\cite{Y} and \cite{JWY} demonstrate that, even in dimension $n=1$, the Hamiltonians $H(p,x)=p^2 + \sin x$ and $H(p,x) = (2-\sin^2 x)^{-1} (p^2-1)$ exhibit multiple smooth absolute minimizers with the same boundary values. For both of these, the corresponding $\overline H$ has a flat spot. What is more, the sufficient condition we outline below for $\argmin \overline H$ to have empty interior is the same condition conjectured in \cite{JWY} to be sufficient for the comparison principle to hold for absolute minimizers of $H$.

The crude two-sided bound
\begin{equation} \label{crude2side}
\essinf_{\omega \in \Omega} H(p,0,\omega) \leq \overline H(p) \leq \esssup_{\omega \in \Omega} H(p,0,\omega),
\end{equation}
which follows from stationarity and the bound $\inf_{y\in\Rd} H(p,y) \leq \overline H(p,y) \leq \sup_{y\in\Rd} H(p,y)$ proved in Section~\ref{bigdist}, provides the following simple sufficient condition for $\overline H$ to have no flat spot: 
\begin{equation} \label{Hbarnoflat}
\begin{cases}
\text{there exists  some closed $\Gamma \subseteq \Rd$ with empty interior such that} \\[2mm] 
 \ H \equiv h_0 \ \ \mbox{on} \ \ \Gamma \times \Rd \times \Omega \ \  \mbox{and} \ \ h_0 < \essinf_{y\in\Rd} H(p,y,\omega) \ \ \mbox{for all} \ p  \in \Rd\setminus \Gamma.
\end{cases}
\end{equation}
Indeed, it is clear from \eqref{crude2side} that \eqref{Hbarnoflat} implies $\argmin \overline H = \Gamma$.

An example of an explicit Hamiltonian satisfying \eqref{Hbarnoflat}, in this case with $\Gamma=\{ 0 \}$, is
\begin{equation} \label{exameq}
H(p,y,\omega) =  \frac{p\cdot a(y,\omega)p}{2|p|},
\end{equation}
where $a$ is a stationary process with values in the positive matrices and $\lambda \leq a(y,\omega) \leq \Lambda$ for all $(y,\omega)\in \Rd\times\Omega$.

A Hamiltonian with no flat spot but which does not satisfy \eqref{Hbarnoflat} is
\begin{equation}\label{refereesug}
H(p,y,\omega) : =  \frac12 |p|^2 + b(y,\omega)\cdot p,
\end{equation}
where, in addition to being stationary ergodic, Lipschitz and bounded, the vector field $b$ satisfies the mean-zero and divergence-free condition
\begin{equation}\label{dzmz}
\mathrm{div}\, b \equiv 0 \quad \mbox{and} \quad \E[ \,b(0,\cdot)] = 0.
\end{equation}
It is obvious that $\overline H(0) = 0$. We will demonstrate the lack of a flat spot by showing that
\begin{equation}\label{flowc}
\overline H(p) \geq \frac12 |p|^2
\end{equation}
To prove \eqref{flowc}, select a nonnegative smooth cutoff function $\varphi_\delta$ which has support in $B_{R/\delta}$ for large $R>0$ and such that $\int_{\Rd} \varphi_\delta(y) \, dy = 1$ and $\int_{\Rd} |D\varphi_\delta| \, dy \leq \delta/R$. It is possible to choose, for example, a suitable multiple of (a regularization of) the function $\varphi_\delta(y) : = (R^2\delta^{-2} - |y|^2)_+$. Multiplying \eqref{mac} by $\varphi_\delta$ and integrating over $\Rd \times \Omega$, we have, after an integration by parts and in view of \eqref{dzmz}:
\begin{align*}
0 = \E \big[ \delta v^\delta \big] + \frac12 \E\big[ |p+Dv^\delta|^2 \big] - \E \int_{\Rd} v^\delta b\cdot D\varphi_\delta\, dy.
\end{align*}
Using Jensen's inequality and passing to the limit $\delta \to 0$ we obtain, using the results in Section~\ref{EH},
\begin{align*}
0 \geq -\overline H(p) + \frac12 |p|^2  - C\overline |H(p)| \limsup_{\delta\to 0} \frac1\delta \int_{\Rd} D\varphi_\delta\, dy \geq -\overline H(p) + \frac12 |p|^2  - \frac{C}{R} |\overline H(p)|.
\end{align*}
A rearrangement of this expression yields \eqref{flowc} after sending $R\to \infty$.

\section{The proofs of the results about the distance functions}
\label{MP}

We begin with a comparison principle for \eqref{EeqX} in exterior domains for $\mu > \overline H(p)$, following an argument introduced very recently by the authors in \cite{ASo}. Its main feature, which makes it quite unusual when compared to comparison results found in the literature, is that it is not assumed that the subsolution and supersolution separate at most strictly sublinearly from each other at infinity. Indeed, we merely require the negative part of the supersolution to be strictly sublinear at infinity and the subsolution to grow no fast than $\sim|x|$.

In the proof, we lower the subsolution until it has strictly sublinear separation from $v$, and apply the usual comparison principle for Hamilton-Jacobi equations. It is then shown that, if we had lowered $u$ at all, then we could have lowered it a bit less-- and therefore we need not have lowered it at all. To prove the latter, a term $\varphi_R$, which is small in balls of radius $\sim R$ but grows linearly at infinity for each fixed $R$, is subtracted from the subsolution. We then compare the result with $v$, and then conclude sending $R \to \infty$. The fact that the parameter $\mu$ is strictly larger than $\overline H(p)$ permits us to compensate for this perturbation with the use of a global subsolution of \eqref{EeqX}.

\begin{prop} \label{metcomp}
Assume that $H:\Rd\times\Rd\to\R$ satisfies \eqref{Hconvex}, \eqref{Hcoer} and \eqref{Huc}. Fix $p\in \Rd$, $\mu > \overline H(p)$ and a compact subset $D\subseteq \Rd$. Suppose that $u \in \USC(\overline{\Rd\!\setminus \!D})$ is a subsolution of \eqref{EeqX}, $v \in \LSC(\overline{\Rd\!\setminus \!D})$ is a supersolution of \eqref{EeqX}, and
\begin{equation} \label{meteqcmpgc}
\liminf_{|x|\to \infty} \frac{v(x)}{|x|} > 0 \quad \mbox{and} \quad \limsup_{|x|\to \infty} \frac{u(x)}{|x|} < \infty.
\end{equation}
Then
\begin{equation} \label{metcmpcon}
\sup_{\Rd \setminus D} (u-v) = \sup_{\partial D} (u-v).
\end{equation}
\end{prop}
\begin{proof}
Since $p$ plays no role, we may suppose for simplicity that $p=0$. We may also assume $\limsup_{|x| \to \infty} u(x)/|x| \geq 0$, since otherwise the result is immediate from the classical comparison principle. Define
\begin{equation*}
\Lambda : = \Big\{ 0 \leq \lambda \leq 1 : \liminf_{|x| \to \infty} \frac{v(x) - \lambda u(x)}{|x|} \geq 0 \Big\} \qquad \mbox{and} \qquad \overline \lambda: = \sup \Lambda.
\end{equation*}
 The assumption \eqref{meteqcmpgc} implies that $[0,\beta) \in \Lambda$ for some $\beta > 0$, and hence $\overline \lambda > 0$. We next show that $\Lambda = \big[0,\overline \lambda \big]$. To see that $\overline \lambda \in \Lambda$, select $\ep > 0$ and $\lambda \in \Lambda$ with $\overline \lambda \leq \lambda + \ep$ and observe that by \eqref{meteqcmpgc},
\begin{multline*}
\liminf_{|x| \to \infty} \frac{v(x) - \overline \lambda u(x)}{|x|}  \geq \frac{\overline \lambda}{\lambda+\ep} \liminf_{|x| \to \infty} \frac{v(x) - (\lambda+\ep)u(x)}{|x|} \\ \geq \frac{\overline \lambda}{\lambda+\ep} \left( -\ep \limsup_{|x| \to\infty} \frac{u(x)}{|x|}\right) \geq -C\overline \lambda \ep(\lambda+\ep)^{-1}.
\end{multline*}
Sending $\ep \to 0$ yields $\overline \lambda \in \Lambda$. If $\lambda \in \big(0,\overline \lambda\big)$, then using again \eqref{meteqcmpgc}, we have
\begin{equation*}
\liminf_{|x| \to \infty} \frac{v(x) - \lambda u(x)}{|x|}  \geq \frac{\lambda}{\overline \lambda} \liminf_{|x| \to \infty} \frac{v(x) - \overline \lambda u(x)}{|x|} \geq 0.
\end{equation*}
The claim is proved.

We claim that $\overline \lambda = 1$. Select $\lambda \in \Lambda$ with $0<\lambda < 1$. For each $R> 1$, define the auxiliary function
\begin{equation} \label{varphiR}
\varphi_R(x): = R- (R^2+|x|^2)^{1/2},
\end{equation}
and observe that, for a constant $C > 0$ independent of $R> 1$,
\begin{equation} \label{phidcn}
\sup_{x\in \Rd} |D\varphi_R(x)| \leq C.
\end{equation}
We have defined $\varphi_R$ in such a way that $-\varphi_R$ grows at a linear rate at infinity, which is independent of $R$, while $\varphi_R \to 0$ as $R \to \infty$. Indeed, it is easy to check that
\begin{equation} \label{phideath}
|\varphi_R(x)| \leq |x|^2 \left( R^2 + |x|^2 \right)^{-1/2}.
\end{equation}
Fix constants $0 < \eta < 1$ and $\theta > 1$ to be selected below. By \eqref{Huc} and \eqref{phidcn}, we have
\begin{equation} \label{Ephi}
H(-\theta D\varphi_R,x) \leq C_\theta \quad \mbox{in} \ \Rd.
\end{equation}
Define the function
\begin{equation*}
\hat u : = \left( \lambda + \eta\right) u+ (1-\lambda -\eta) w
\end{equation*}
as well as
\begin{align*}
\hat u_R  : = &\ (1-\eta) \hat u + \eta \theta \varphi_R = (1-\eta)(\lambda+\eta) u + (1-\eta)(1 - \lambda- \eta)w + \eta \theta \varphi_R,
\end{align*}
where $w$ is the function in assumption \eqref{ssass}. By subtracting a constant from $w$, we may assume that $\sup_{D} w = 0$. Since $0<\lambda < 1$, we may select $\eta>0$ small enough, depending only on a positive lower bound for $1-\lambda$, that
\begin{equation} \label{etachoice}
\lambda +\eta < 1 \quad \mbox{and} \quad (1-\eta)(\lambda +\eta) > \lambda.
\end{equation}
Select $\theta : = 1+ \limsup_{|x| \to \infty} u(x) / |x|$, and observe that, by the previous inequality, the sublinearity of $w$ at infinity, $\lambda \in \Lambda$ and the definition of $\varphi_R$, we have, for every $R> 1$,
\begin{align*}
\liminf_{|y| \to \infty} \frac{v(y) - \hat u_R(y)}{|y|}  & \geq \liminf_{|y| \to \infty} \frac{v(y) - \lambda u(y)}{|y|} + \liminf_{|y| \to \infty} \frac{\lambda u - \hat u_R(y)}{|y|} \\
& \geq 0 + \eta \liminf_{|y| \to \infty} \frac{(\eta-(1-\lambda))u(y)  -  \theta \varphi_R(y)}{|y|}\\
& \geq - \eta(1-\lambda-\eta) \theta +\eta\theta \\
& > 0.
\end{align*}
To get a differential inequality for $\hat u_R$, we apply Lemma~\ref{convtrick} twice. The first application, using~\eqref{ssass} and that $u$ is a subsolution of \eqref{Eeq}, yields that $\hat u$ satisfies
\begin{equation*}
H(D\hat u,x) \leq (\lambda+\eta) \mu + (1-\lambda-\eta) \mu_0 \quad \mbox{in} \ \Rd\! \setminus \! D.
\end{equation*}
Combining this with \eqref{Ephi}, we obtain
\begin{equation*}
H(D\hat u_R,y) \leq \widetilde\mu(\eta) \quad \mbox{in} \ \Rd\! \setminus \! D,
\end{equation*}
where the constant $\widetilde\mu(\eta)$ is given by
\begin{equation*}
\widetilde \mu(\eta) := (1-\eta)(\lambda+\eta)\mu + (1-\eta)(1-\lambda-\eta)\mu_0 +  C\eta.
\end{equation*}
Since $\mu > \mu_0$ and $\lambda$, it is possible to select $\eta> 0$ sufficiently small, depending on a positive lower bound for $\lambda$, so that $\widetilde\mu(\eta) < \mu$. The classical comparison principle then applies, yielding
\begin{equation*}
\hat u_R - v \leq \max_{\partial D} (\hat u_R - v) \quad \mbox{in} \ \Rd\!\setminus \! D.
\end{equation*}
Sending $R\to \infty$ and using the fact that $\varphi_R\rightarrow 0$ locally uniformly, we obtain
\begin{equation} \label{hatuvinq}
(1-\eta) \hat u - v  \leq \max_{\partial D} \big( (1-\eta) \hat u - v\big) \quad \mbox{in} \ \Rd\! \setminus\! D.
\end{equation}
Since $w$ is strictly sublinear at infinity, the latter implies
\begin{equation*}
\liminf_{|y| \to \infty} \frac{v(y) - (1-\eta)(\lambda+\eta)u(y)}{|y|} \geq 0.
\end{equation*}
Hence
\begin{equation*}
\overline \lambda \geq (1-\eta)(\lambda+\eta).
\end{equation*}
If $\overline \lambda < 1$, then we may send $\lambda \to \overline \lambda$ while keeping $\eta > 0$ fixed to find that to obtain that $\overline \lambda \geq  (1-\eta)(\overline \lambda +\eta)$, which is a contradiction for small enough $\eta> 0$. It follows that $\overline \lambda =1$.

We therefore have \eqref{hatuvinq} for every $0 < \lambda < \overline \lambda = 1$ and sufficiently small $\eta > 0$, depending on $\lambda$. Sending $\eta \to 0$ and then $\lambda \to 1$ in \eqref{hatuvinq} yields \eqref{metcmpcon}.
\end{proof}

Using Perron's method we show that solutions of \eqref{EeqX} satisfying appropriate growth conditions exist, completing the proof of Proposition~\ref{metexistence}.

\begin{proof}[{Proof of Proposition~\ref{metexistence}}]
Fix $p\in \Rd$, $\mu > \overline H(p)$ and $x_0\in \Rd$. According to \eqref{Hcoer}, for large enough $\alpha > 0$, the function $v(x):= \alpha|x-x_0|$ is a strict supersolution of \eqref{EeqX} in $\Rd \setminus \{ x_0 \}$.

There exists a global subsolution $w$ of $H(p+Dw,y) \leq \overline H(p)$ which is strictly sublinear at infinity and globally Lipschitz. Since $\mu > \overline H(p)$, it follows from \eqref{Huc} that for small enough $\ep > 0$ the function
\begin{equation*}
\hat w(y) = w(y) + \ep (1+|y|^2)^{1/2}
\end{equation*}
satisfies $H(p+D\hat w,y) \leq \mu$. By subtracting a constant we may assume $\hat w(x_0) = 0$. Define
\begin{equation*}
d_{x_0,\mu,p}(x):= \sup\big\{ u(x): u \in \USC(\Rd) \ \mbox{satisfies} \ H(p+Du,x) \leq \mu \ \mbox{in} \ \Rd\setminus \{ x_0 \}, \ \ u \leq v \ \ \mbox{in} \ \Rd \big\}.
\end{equation*}
By \eqref{ssass}, the classical Perron method adapted to viscosity solutions and Proposition~\ref{metcomp}, we have that $d_{x_0,\mu}$ is a solution of \eqref{EeqX}. From $d_{x_0,\mu,p} \geq \hat w$, \eqref{EeqBC} follows. Uniqueness is immediate from Proposition~\ref{metcomp}.
\end{proof}

\begin{remark} \label{redud}
For every $p,q\in \Rd$, $\mu > \max\big\{ \overline H(p), \overline H(q)\big\}$ and $x_0\in \Rd$,
\begin{equation} \label{redudeq}
d_{\mu,{x_0},p} (y) = (q-p)\cdot (y-x_0) + d_{\mu,{x_0},q}(y).
\end{equation}
To see this, notice that the right side of \eqref{redudeq} grows at most linearly at infinity and is a solution of \eqref{EeqX} in $\Rd \setminus \{x_0\}$. Therefore $(q-p)\cdot (y-x_0) + d_{\mu,{x_0},q}(y) \leq d_{\mu,{x_0},p}(y)$ by Proposition~\ref{metcomp}. The reverse inequality is obtained by interchanging $p$ and $q$ and repeating the argument.
\end{remark}

In light of \eqref{redudeq}, the family $\big\{ d_{\mu,x_0,p} : p\in \Rd, \ \mu > \overline H(p) \big\}$ is completely described by the single function $d_{\mu,{x_0},{p_*}}$ for any fixed $p_* \in \argmin \overline H$. This motivates the definition \eqref{distfundef} of the distance functions $d_{\mu,{x_0}}$ given in Section~\ref{bigdist}, which we remark does not depend on the choice of $p_*$.

We conclude with a simple proof the necessity of comparisons with distance functions for absolute minimizers. The argument is based on Proposition~\ref{metcomp}.

\begin{proof}[{Proof of Proposition~\ref{CCprop}}]
Suppose that $u\in C^{0,1}_{\mathrm{loc}}(U)$ is absolutely subminimizing but there exists $V \ll U$, $x_0 \in \Rd \setminus V$, and $\mu > \mu_0$ such that
\begin{equation} \label{violcc}
\sup_{V} \big(u - d_{x_0,\mu}\big) > \sup_{\partial V} \big(u - d_{x_0,\mu}\big).
\end{equation}
By subtracting a constant from $u$, if necessary, we may assume that $u < d_{x_0,\mu}$ on $\partial V$ but $u > d_{x_0,\mu}$ at some point of $V$. Define $W:= \{ x\in V: u(x) > d_{x_0,\mu}(x) \}$. Since $u$ is absolutely subminimizing, it follows that
\begin{equation*}
\esssup_{x\in W} H(Du(x),x) \leq \esssup_{x\in W} H(Dd_{x_0,\mu}(x),x) = \mu.
\end{equation*}
Thus $u$ is a subsolution of \eqref{Eeq} in $W$. Let $\ep > 0$. The function
\begin{equation*}
v(x) : = \begin{cases} \max\{ u(x)-\ep, d_{x_0,\mu}(x) \} & x\in V, \\
d_{x_0,\mu}(x) & x\in \Rd\setminus D.
\end{cases}
\end{equation*}
is a subsolution of \eqref{Eeq} in $\Rd\setminus \{ x_0 \}$. It follows from Proposition~\ref{metcomp} that $v\leq d_{x_0,\mu}$ in $\Rd \setminus \{ x_0 \}$, a contradiction to \eqref{violcc} for $\ep > 0$ small enough.
\end{proof}

\begin{remark}\label{remkey}
In the case that $H$ is independent of $y$, the condition $\mu \geq H(p)$ is sharp for the existence of distance functions. That is, there are no nonnegative solutions of \eqref{Eeq} in $\Rd\setminus \{ 0 \}$, provided $\mu < H(p)$. Indeed, suppose on the contrary that such a function $u$ exists. Then $u$ is Lipschitz by Lemma~\ref{convtrick}, and by considering any point of differentiability, we deduce that $\mu \geq \min H$. Choosing $\mu < \nu < H(p)$, we may apply Proposition~\ref{metcomp} to deduce that
\begin{equation*}
u(y) \leq d_{\nu}(y) - p\cdot y = \max\big\{ q\cdot y\, : \, H(q+p) \leq \nu \big\}.
\end{equation*}
The convexity of the sublevel sets of $H$ and $0 \not \in \{ q\, : \, H(q+p) \leq \mu\}$ yield, via convex separation, a contradiction to $u\geq 0$.
\end{remark}

\section{The macroscopic problem} \label{EH}

The classical method for identifying the effective equation in the homogenization of Hamilton-Jacobi equations begins with the consideration of the \emph{macroscopic problem}
\begin{equation} \label{mac}
\delta v^\delta + H(p+Dv^\delta, y, \omega) = 0 \quad \mbox{in} \ \Rd.
\end{equation}
Here $\delta > 0$ and $p\in \Rd$ are fixed. We will see shortly that \eqref{mac} has a unique bounded solution $v^\delta = v^\delta(y,\omega;p)$ which is globally Lipschitz continuous. The functions $v^\delta$ are sometimes called \emph{approximate correctors}, and, in the context of periodic homogenization, \eqref{mac} approximates the cell problem. The effective Hamiltonian is typically constructed as a limit (in an appropriate sense), as $\delta \to 0$, of $-\delta v^\delta(0,\omega;p)$, which is shown to have a limit with the help of the ergodic theorem.

The next proposition establishes the well-posedness of \eqref{mac}. Since it is well-known, we merely sketch the proof. Further details may be found for example in \cite{CIL}.

\begin{prop} \label{mac-wp}
For each $\delta > 0$, $p\in \Rd$ and $\omega \in \Omega$, there exists a unique bounded solution $v^\delta = v^\delta(\cdot,\omega;p)$ of \eqref{mac}. Moreover, the map $(y,\omega) \mapsto v^\delta(y,\omega;p)$ is stationary, and there exists a constant $C = C(|p|)>0$ such that
\begin{equation} \label{vdelest}
\sup_{(y,\omega) \in \Rd\times \Omega} \Big( \big| \delta  v^\delta(y,\omega;p) \big| + \big| Dv^\delta(y,\omega;p)\big| \Big) \leq C.
\end{equation}
\end{prop}
\begin{proof}
For $\delta > 0$, the classical comparison principle applies to \eqref{mac}, allowing us to compare subsolutions and supersolutions which separate at most strictly sublinearly at infinity. According to \eqref{Hec}, the constant (in $y$) functions $-C_1(p,\omega)/ \delta$ and $-C_2(p,\omega)/\delta$ are a supersolution and subsolution of \eqref{mac}, respectively, where
\begin{equation} \label{Cidef}
C_1(p,\omega):= \inf_{y \in \Rd} H(p,y,\omega) \leq \sup_{y \in \Rd} H(p,y,\omega) =: C_2(p,\omega).
\end{equation}
Notice that, by \eqref{Hec}, we have $|C_1(p,\omega)| + |C_2(p,\omega)| \leq C(|p|)$, and, in view of \eqref{Hcoe} and \eqref{Hec}, for each $\omega\in \Omega$, we find
\begin{equation} \label{C1coe}
C \leq C_1(p,\omega) \rightarrow +\infty \quad \mbox{as} \ \ |p| \to \infty.
\end{equation}

The Perron method now provides the existence of a solution $v^\delta$ of \eqref{mac}, satisfying
\begin{equation} \label{vdelCi}
C_1(p,\omega) \leq -\delta v^\delta(y,\omega;p) \leq C_2(p,\omega),
\end{equation}
which implies $\big| \delta v^\delta \big| \leq C(|p|)$. By comparison, $v^\delta$ is the unique solution which grows at most sublinearly at infinity. By uniqueness and \eqref{Hstat}, $v^\delta$ is stationary. Using the equation, the bound for $\big| \delta v^\delta \big|$, \eqref{Hcoe} and Lemma~\ref{convtrick} yield that $\big| Dv^\delta \big|$ is uniformly bounded.
\end{proof}

Our goal is to characterize, in the limit $\delta \to 0$, the behavior of the functions $\delta v^\delta(\cdot,\cdot;p)$ on a set of full probability and simultaneously for all $p\in \Rd$. To accomplish this, we typically characterize the limit for each fixed $p \in \Q^\d$ and then take the (countable) intersection of the resulting subsets of $\Omega$. To conclude, we need a continuous dependence estimate. This is the purpose of the next lemma.

\begin{lem} \label{CDE-noxd}
For each $R> 0$, there exists a constant $C_R$ such that, for every $\delta > 0$, $y\in \Rd$, $p,q\in B_R$, and $\omega \in \Omega$,
\begin{equation} \label{CDEeq}
\big| \delta v^\delta(y,\omega;p) - \delta v^\delta(y,\omega;q) \big| \leq C_R|p-q|.
\end{equation}
\end{lem}
\begin{proof}
We may assume that $|p-q| < 1$. Fix $k> 0$ to be selected and consider the function
\begin{equation*}
u(x): = (1-|p-q|) v^\delta(y,\omega;q) - \delta^{-1} k.
\end{equation*}
Using the convexity of $H$, we find that, formally,
\begin{align*}
\delta u + H(p+Du,y,\omega) & \leq -k+|p-q| H\Big(\frac{p-q+|p-q|q}{|p-q|},y,\omega\Big) \\
& \leq -k + |p-q| \sup_{B_{R+1}\times \Rd} H(\cdot,\cdot,\omega).
\end{align*}
Choosing $k:= |p-q| \sup_{B_{R+1}\times \Rd} H(\cdot,\cdot,\omega)$, find that $u$ is formally a subsolution of \eqref{mac}. This calculation is made rigorous with the help of Lemma~\ref{convtrick}, and so we deduce that $u(y) \leq v^\delta(y,\omega;p)$. This yields, with the help of \eqref{vdelest},
\begin{equation*}
 \delta v^\delta(y,\omega;q) -  \delta v^\delta(y,\omega;p) \leq k + C(|q|) |p-q| \leq C(|q|) |p-q|.
\end{equation*}
Repeating the argument with the roles of $p$ and $q$ reversed yields \eqref{CDEeq}.
\end{proof}

\begin{lem}
For every $\delta > 0$, $y\in \Rd$, $p,q\in B_R$, and $\omega \in \Omega$,
\begin{equation} \label{vdelconv}
 \frac12  v^\delta(y,\omega;p) + \frac12  v^\delta(y,\omega;p) \leq   v^\delta(y,\omega;\tfrac12 p+\tfrac12 q).
\end{equation}
\end{lem}
\begin{proof}
Fix $\omega$. The convexity of $H$ implies that the left side of \eqref{vdelconv} is a subsolution of
\begin{equation*}
\delta u + H\big( \tfrac12 p + \tfrac12 q + Du, y, \omega\big) \leq 0 \quad \mbox{in} \ \Rd.
\end{equation*}
Therefore by the comparison principle, $u(\cdot) \leq v^\delta (\cdot,\omega;\tfrac12 p+\tfrac12 q)$.
\end{proof}

The next proposition, which plays an important role in the proof of Theorem~\ref{main}, characterizes the limit of $-\delta v^\delta$. The argument first appeared in \cite{LS} in the context of the homogenization of Hamilton-Jacobi equations, and we give a full proof here for completeness. The idea is to exploit the convexity of $H$ to find a \emph{subcorrector} $w$ which has a stationary gradient and a minimal constant $\widetilde H(p)$ such that
\begin{equation*}
H(p+Dw,y,\omega) \leq \widetilde H(p) \quad \mbox{in} \ \Rd.
\end{equation*}
This permits us to apply one-sided comparison arguments which are enough to conclude that $-\delta v^\delta(0,\omega;p)$ converges to the constant $\widetilde H(p)$ in probability. In the next section we prove that $\widetilde H(p) = \overline H(p)$, a fact not immediately obvious since $\overline H$ is defined ``$\omega$-by-$\omega$" while $\widetilde H$ is defined in terms of stationary functions.

We remark that obtaining the almost sure convergence of $-\delta v^\delta(0,\omega;p)$ to $\widetilde H(p)$ is considerably more involved; see \cite{ASo} for a more detailed overview and further discussion.

\begin{prop} \label{mainstep}
There exists $\widetilde H:\Rd\to \R$ which is continuous, convex, and coercive, such that, for every $R> 0$ and $p \in \Rd$,
\begin{equation} \label{mainstepeq}
\lim_{\delta \to 0} \E \Big[ \, \sup_{y\in B_{R/\delta}} \left| \delta v^\delta (y,\omega;p) + \widetilde H(p)\right| \Big]=0.
\end{equation}
Moreover, there exists a subset $\Omega_2 \subseteq \Omega$ of full probability such that, for every $\omega \in \Omega_2$,
\begin{equation} \label{oneside}
\liminf_{\delta \to 0} \delta v^\delta(0,\omega) = -\widetilde H(p).
\end{equation}
\end{prop}
\begin{proof}
The local Lipschitz continuity of $\widetilde H$ follows from Lemma~\ref{CDE-noxd}, once we have shown \eqref{mainstepeq}. Likewise, the coercivity and convexity follow from \eqref{C1coe}, \eqref{vdelCi} and \eqref{vdelconv}. We fix $p\in \Rd$, and omit all dependence of $p$. The proof is divided into three steps.

\emph{Step 1: Construction of the subcorrector.}
For each $\delta > 0$, define
\begin{equation*}
w^\delta(y,\omega) := v^\delta(y,\omega) - v^\delta(0,\omega).
\end{equation*}
According to \eqref{vdelest}, there exists a subsequence $\delta_j \to 0$, a random variable $\widetilde H = \widetilde H(p,\omega) \in \R$, a function $w\in L^\infty(\Rd\times \Omega)$ and a field $\Phi \in L^\infty(\Rd\times\Omega ; \Rd)$ such that, for every $R> 0$, we have the following limits as $j\to \infty$:
\begin{equation} \label{weaklims}
\left\{ \begin{aligned}
& -\delta_j v^{\delta_j}(0,\cdot ) \rightharpoonup \widetilde H(p,\cdot) \quad \mbox{weakly-}\!\ast \ \mbox{in} \ L^\infty(\Omega), \\
& w^{\delta_j} \rightharpoonup w \quad \mbox{weakly-}\!\ast \ \mbox{in} \ L^\infty(B_R\times\Omega), \\
& Dv^{\delta_j} \rightharpoonup \Phi \quad \mbox{weakly-}\!\ast \ \mbox{in}  \ L^{\infty}(B_R\times\Omega;\Rd).
\end{aligned} \right.
\end{equation}
The stationarity of the functions $v^{\delta_j}$, the ergodicity hypothesis and the Lipschitz estimate \eqref{vdelest} imply that $\widetilde H$ is independent of $\omega$, i.e., $\widetilde H(p,\omega) = \widetilde H(p)$ a.s. in $\omega$. Indeed, it suffices to check that, for each $\mu \in \R$, the event $\{ \omega \in \Omega : \widetilde H(p,\omega) \geq \mu \}$ is invariant under $\tau_y$, which follows immediately from \eqref{vdelest}.

The vector field $\Phi$ inherits stationarity from the sequence $Dv^{\delta_j}$ and is gradient-like in the sense that, for every compactly-supported smooth test function $\psi=\psi(y)$,
\begin{equation*}
\int_{\Rd} \left( \Phi^i(y,\omega) \psi_{y_j} (y) - \Phi^j(y,\omega) \psi_{y_i}(y) \right) \, dy = 0 \quad \mbox{a.s. in} \ \omega.
\end{equation*}
It follows from Lemma~\ref{koz} that $\Phi = Dw$, a.s. in $\omega$, in the sense of distributions and, moreover, that $w(\cdot,\omega)$ is globally Lipschitz a.s. in $\omega$.

The convexity hypothesis \eqref{Hconv} and the equivalence of distributional and viscosity solutions for linear inequalities (c.f. Ishii~\cite{I}) allow us to pass to weak limits in \eqref{mac} obtaining that $w(\cdot,\omega)$ is a viscosity solution, a.s. in $\omega$, of
\begin{equation} \label{subcorreq}
 H(p+Dw, y, \omega) \leq \widetilde H(p) \quad \mbox{in} \ \Rd.
\end{equation}
By Lemma~\ref{koz} and
\begin{equation*}
\E\Phi(0,\cdot) = \lim_{j\to \infty} \E Dv^{\delta_j}(0,\cdot) = 0,
\end{equation*}
we have that $w$ is strictly sublinear at infinity, that is,
\begin{equation} \label{sublininfty}
\lim_{|y|\to\infty} |y|^{-1}w(y,\omega) = 0 \quad \mbox{a.s. in} \ \omega.
\end{equation}

\emph{Step 2: Show that $\widetilde H$ characterizes the full limit of $\delta v^\delta(0,\omega)$ in $L^1(\Omega,\mathbb P)$.}
We first prove
\begin{equation} \label{cmpineq}
-\widetilde H(p) \leq \liminf_{\delta \to 0} \delta v^\delta(0,\omega) \quad \mbox{a.s. in} \ \omega.
\end{equation}
This is done via a one-sided comparison argument, using the subcorrector $w$ to bound $v^\delta$ from below. Let $\Omega_1$ be a subset of $\Omega$ with $\Prob[\Omega_1] = 1$ such that for every $\omega \in \Omega_1$, we have $\widetilde H(p,\omega) = \widetilde H(p)$ as well as \eqref{subcorreq} and \eqref{sublininfty}.

Fix $\omega\in \Omega_1$ and a small constant $\eta > 0$. We allow the constants introduced immediately below to depend on $\omega$. Define $\varphi(y):= -(1+|y|^2)^{1/2}$, and notice that \eqref{Hec} implies
\begin{equation} \label{msphibnd}
H(p+D\varphi,y,\omega) \leq C.
\end{equation}
For each $\delta > 0$, define the function
\begin{equation*}
w^\delta(y) : = (1-\ep) \!\left( w(y,\omega) - ( \widetilde H + \eta ) \delta^{-1} \right) + \ep \varphi(y),
\end{equation*}
where $\ep > 0$ will be chosen below in terms of $\eta$. We proceed by comparing $w^\delta$ and $v^\delta$ in the limit as $\delta \to 0$. Assuming that $w$ is smooth, we have, in view of \eqref{subcorreq}, \eqref{msphibnd} and the convexity of $H$,
\begin{equation}
\delta w^\delta+ H(p+Dw^\delta,y)  \leq \delta w^\delta + (1-\ep) \widetilde H + C \ep.
\label{delwdel}
\end{equation}
In the case that $w$ is not smooth, we verify \eqref{delwdel} in the viscosity sense either by using that $\varphi$ is smooth, or by appealing to Lemma~\ref{convtrick}. According to \eqref{sublininfty},
\begin{equation*}
\sup_{B_R} w \leq C_\eta + \eta^3 R,
\end{equation*}
and so, by choosing $\ep = \min\{ 1/4, \eta/4C\}$ with $C$ is as in \eqref{delwdel}, we may estimate the right side of \eqref{delwdel} by
\begin{equation*}
\delta w^\delta + (1-\ep) \widetilde H + C \ep = (1-\ep) (\delta w-\eta) + C\ep \leq \delta C_{\eta} + \delta \eta^3 R - \frac12 \eta \quad \mbox{in} \ B_R.
\end{equation*}
Next we observe that the bound on $\big|\delta v^\delta\big|$ in \eqref{vdelest} and the definition of $w^\delta$ imply
\begin{equation*}
w^\delta - v^\delta \leq (1-\ep) w +C \delta^{-1} - c\eta R \quad \mbox{on} \ \partial B_R.
\end{equation*}
Therefore by taking $R : = C(\delta\eta)^{-1}$ for a large constant $C> 0$, we have for all sufficiently small $\delta > 0$, depending on both $\omega$ and $\eta$,
\begin{equation}
\left\{ \begin{aligned}
& \delta w^\delta + H(p+Dw^\delta,x,y,\omega) \leq 0 & \mbox{in} & \ B_R, \\
& w^\delta \leq v^\delta & \mbox{on} & \ \partial B_R.
\end{aligned} \right.
\end{equation}
The comparison principle yields that $w^\delta(\cdot) \leq v^\delta(\cdot,\omega)$ in $B_R$. In particular we deduce that $w^\delta(0) \leq v^\delta(0,\omega)$. Multiplying this inequality by $\delta$ and sending $\delta \to 0$ yields
\begin{equation*}
-\widetilde H -\eta \leq ( 1 - C \eta)^{-1} \liminf_{\delta \to 0} \delta v^\delta (0,\omega).
\end{equation*}
Disposing of $\eta > 0$, we have \eqref{cmpineq} for all $\omega\in \Omega_1$.

Since $-\widetilde H$ is the weak limit of the subsequence $\delta_j v^{\delta_j}(0,\cdot)$, the reverse of \eqref{cmpineq} is immediate and we obtain \eqref{oneside} for an event $\Omega_2 \subseteq \Omega$ of full probability. Furthermore, it follows from this that the full sequence $\delta v^\delta(0,\cdot)$ converges weakly-$\ast$ to $-\widetilde H$, that is, as $\delta \to 0$,
\begin{equation*}
\delta v^\delta (0,\cdot) \rightharpoonup -\widetilde H  \quad \mbox{weakly-}\!\ast \ \mbox{in} \ L^\infty(\Omega).
\end{equation*}
An application of Lemma~\ref{meastheor} yields
\begin{equation}\label{convprob}
\lim_{\delta \to 0} \E \big| \delta v^\delta(0,\cdot) + \widetilde H \big| =0.
\end{equation}

\emph{Step 3: Improve \eqref{convprob} to balls of radius $\sim1/\delta$.} We show that, for each $R> 0$,
\begin{equation} \label{convprob1del}
\lim_{\delta\to0} \E \bigg[ \sup_{y\in B_{R/\delta}} \big| \delta v^\delta(0,\cdot) + \widetilde H \big| \bigg] = 0.
\end{equation}
Fix $R > 0$. Let $\rho > 0$ and select points $y_1,\ldots,y_k \in B_R$ such that
\begin{equation*}
B_R \subseteq \bigcup_{i=1}^k B(y_i,\rho) \quad \mbox{and} \quad k \leq C \bigg( \frac{R}{\rho}\bigg)^\d.
\end{equation*}
Using \eqref{vdelest}, we find
\begin{align*}
\lefteqn{ \limsup_{\delta \to 0} \E \bigg[ \sup_{y\in B_{R/\delta}} \big| \delta v^\delta(0,\cdot) + \widetilde H \big| \bigg]} \qquad & \\
& \leq \sum_{i=1}^k \limsup_{\delta \to 0} \E \big| \delta v^\delta(y_i/\delta ,\cdot) + \widetilde H \big| + \limsup_{\delta\to 0} \E\bigg[  \max_{1\leq i \leq k} \osc_{z \in B(y_i/\delta, \rho/\delta)} \delta v^\delta(z,\cdot) \bigg] \\
& \leq k\limsup_{\delta\to 0}\E \big| \delta v^\delta(0,\cdot) + \widetilde H \big| + C\rho \\
& =  C\rho.
\end{align*}
Disposing of $\rho>0$ yields \eqref{convprob1del}.
\end{proof}

\begin{remark} \label{assubseq}
By \eqref{mainstepeq}, we can find a subsequence $\delta_j \to 0$ so that, a.s. in $\omega$,
\begin{equation}\label{assubseqeqp}
\lim_{j \to \infty} \sup_{y\in B_{R/\delta_j}} \left| \delta_j v^{\delta_j} (y,\omega;p) + \widetilde H(p)\right| =0.
\end{equation}
Using \eqref{assubseqeqp} and a diagonalization procedure, and by intersecting the relevant subsets of $\Omega$, we deduce that the existence of a subsequence $\delta_j \to 0$ along which
\begin{equation}
\label{assubseqeq}
\lim_{j \to \infty} \sup_{y\in B_{R/\delta_j}} \left| \delta_j v^{\delta_j} (y,\omega;p) + \widetilde H(p)\right| =0.
\end{equation}
holds, for every $R> 0$ and rational $p\in\Q^\d$, on a single event $\Omega_0$ of full probability. Using Lemma~\ref{CDE-noxd}, we deduce that \eqref{assubseqeq} holds for all $p\in \Rd$ and $\omega\in \Omega_0$.
\end{remark}

\section{The proof of the homogenization of the distance functions} \label{H}

To homogenize the distance functions, we proceed in two steps: first, we use the subadditive ergodic theorem to show that the distance functions have an almost sure limit. Then we identify this limit with the help of \eqref{assubseqeq}.

\begin{proof}[{Proof of Proposition~\ref{conehmg}}]
In light of \eqref{laidout}, we can write the limit \eqref{limhmgdf} as
\begin{equation*}
t^{-1} d_{\mu,t x_0}(ty,\omega) \rightarrow \overline d_\mu(y-x_0) \quad \mbox{as} \ t \to \infty.
\end{equation*}
In order to apply the subadditive ergodic theorem, we must verify that the distance functions are subadditive in the sense that for every $x,y,z\in \Rd$ and a.s. in $\omega$,
\begin{equation} \label{msubadd}
d_{\mu,x}(y,\omega) \leq d_{\mu,z} (y,\omega) + d_{\mu,x} (z,\omega).
\end{equation}
With $w$ denoting the subcorrector constructed in the first step of the proof of Proposition~\ref{mainstep} for $p=p_*$ and using Proposition~\ref{metcomp}, we deduce that, for all $x,y\in \Rd$ and a.s. in $\omega$,
\begin{equation*}
w(y,\omega) - w(x,\omega) \leq p_*\cdot (y-x) + d_{\mu,x} (y,\omega).
\end{equation*}
Interchanging $x$ and $y$ and then adding the resulting inequalities together, we obtain
\begin{equation} \label{subaddpr}
0 \leq d_{\mu,x} (y,\omega) + d_{\mu,y} (x,\omega).
\end{equation}
Thinking of both sides of \eqref{msubadd} as a function of $y$ with $x$ and $z$ fixed, noting that the inequality holds at both $y=x$ and $y=z$ (the former is \eqref{subaddpr} and the latter is obvious), and applying Proposition~\ref{metcomp} with $D = \{ x,z\}$, we obtain \eqref{msubadd}.

We now apply the subadditive ergodic theorem, as stated in Proposition~\ref{SAergthm}, for fixed $\mu$ and $y\in \Rd$, with $Q( [s,t) )(\omega) = d_{\mu,sy}(ty,\omega)$ and $\sigma_t = \tau_{ty}$. We extend $Q(\cdot)(\omega)$ to $\mathcal{J}$ in the obvious way. Using \eqref{distfunstat}, the uniform Lipschitz continuity of the family $\{ d_{\mu,x_0} \, : \, x_0\in \Rd\}$ and \eqref{msubadd}, we easily check that $Q$ is a continuous subadditive process. Proposition~\ref{SAergthm} now provides, for each $y\in \Rd$ and $\mu > \min \overline H$, a random variable $a_\mu(y,\omega)$ such that
\begin{equation} \label{SAcongapp}
t^{-1} d_{\mu,0}(ty,\omega) = a_\mu(y,\omega) \quad \mbox{a.s. in} \ \omega.
\end{equation}

From \eqref{distfunlip} and \eqref{distfunstat} we have that
\begin{equation*}
\limsup_{t \to \infty} t^{-1} d_{\mu,0}(ty,0,\tau_z\omega) = \limsup_{t \to \infty} t^{-1} d_{\mu,z}(ty+z,\omega) = \limsup_{t \to \infty} t^{-1} d_{\mu,0}(ty,0,\omega).
\end{equation*}
Thus the set $\{ \omega \in \Omega_1: \limsup_{t\to \infty} t^{-1} m_\mu(ty,0,\omega) \leq k \}$ is invariant under $\tau_z$, for each $k\in \R$. The ergodic hypothesis implies that $a_\mu$ can be taken independent of $\omega$, i.e., $a_\mu(y,\omega) = a_\mu(y)$. It is clear that $a_\mu$ is positively homogeneous and $a_\mu(0)= 0$. According to \eqref{EeqBC}, we have $a_\mu > 0$ and hence
\begin{equation} \label{amupos}
\inf_{y\neq 0} |y|^{-1} a_\mu(y) > 0.
\end{equation}
Finally, $a_\mu$ is Lipschitz by \eqref{distfunlip}.

To complete the proof that $a_\mu = \overline d_\mu$, we first show that $a_\mu$ is a solution of
\begin{equation}\label{meteq-bar}
\widetilde H( Da_\mu) = \mu \quad \mbox{in} \ \Rd\setminus \{ 0 \},
\end{equation}
and then show that $\widetilde H = \overline H$. Suppose that $\varphi$ is a smooth function and $x_0 \neq 0$ are such that
\begin{equation} \label{Mmuslm}
x \mapsto a_\mu(x) - \varphi(x) \quad \mbox{has a strict local maximum at} \quad x=x_0.
\end{equation}
We show, using the classical perturbed test function method, that
\begin{equation} \label{Mmuwts}
\widetilde H(p+D\varphi(x_0)) \leq \mu.
\end{equation}
 Arguing by contradiction, we assume that $\theta: = \widetilde H(p+D\varphi(x_0)) - \mu > 0$. Take $\{ \delta_j \}$ to be the subsequence described in Remark~\ref{assubseq}, along which we have \eqref{assubseqeq}. Set $p_1 : = p+D\varphi(x_0)$, and define the perturbed test function
\begin{equation*}
\varphi_j(x) : = \varphi(x) + \delta_j v^{\delta_j} \big(\frac x{\delta_j}, \omega; p_1 \big) + \widetilde H(p_1).
\end{equation*}
We claim that, for all sufficiently large $j$ and sufficiently small $r> 0$, $\varphi_j$ satisfies
\begin{equation} \label{vjclmsdf}
H(p+D\varphi_j,x/\delta_j,\omega) \geq \mu + \frac12 \theta \quad \mbox{in} \ B(x_0,r).
\end{equation}
Since $\varphi_j$ is not smooth in general, we verify the inequality in the viscosity sense. To this end, select a smooth function $\psi$ and a point $x_1 \in B(x_0,r)$ at which $\varphi_j - \psi$ has a local minimum. It follows that
\begin{equation*}
y\mapsto  v^{\delta_j}(y,\omega;p_1) - \delta_j^{-1}\big( \psi(\delta_jy) - \varphi(\delta_jy)  \big) \quad \mbox{has a local minimum at} \quad y= x_1/\delta_j.
\end{equation*}
Using the equation for $v^{\delta_j}$, we obtain
\begin{equation} \label{ptfm1p}
 \delta_j v^{\delta_j}(x_1/\delta_j,\omega;p_1) + H(p+D\psi(x_1), x_1/\delta_j , \omega) \geq 0.
 \end{equation}
The observations above yield, for small $r>0$ and large $j$,
\begin{equation*}
H(p+D\psi,x_1/\delta_j,\omega) \geq \overline H(p_1) - \frac12 \theta = \mu + \frac12 \theta.
\end{equation*}
This confirms the claim \eqref{vjclmsdf} in the viscosity sense.

The comparison principle implies that $d^{\delta_j}_{\mu,0} - \varphi_j$ cannot have a local maximum in $B(x_0,r)$. Sending $j\to \infty$ and using \eqref{SAcongapp}, we obtain a contradiction to \eqref{Mmuslm}. This completes the proof that $a_\mu$ is a subsolution of \eqref{meteq-bar}. The argument that $a_\mu$ is a supersolution of \eqref{meteq-bar} is nearly identical, and so is omitted. We conclude that $a_\mu$ is a solution of \eqref{meteq-bar}.

We now show that $\widetilde H = \overline H$. Using Remark~\ref{remkey} and the fact that, for every $\mu > \overline H(p)$,  the function $a_\mu$ is a solution of \eqref{meteq-bar}, we conclude that $\widetilde H (p) \leq \overline H(p)$. The inequality $\overline H(p) \leq \widetilde H(p)$ is clear from the definition of $\overline H$ and the existence of the subcorrector in the proof of Proposition~\ref{mainstep}. Hence $\widetilde H = \overline H$.

Therefore, $a_\mu$ satisfies $\overline H(p+Da_\mu) = \mu$  in $\Rd\setminus \{0 \}$. We have $a_\mu = \overline d_\mu$ by uniqueness.
\end{proof}

\section*{Acknowledgements}
The first author was partially supported by NSF Grant DMS-1004645 and the second author by NSF Grant DMS-0901802. The first author also thanks Charlie Smart and Vesa Julin for helpful conversations. We acknowledge an anonymous referee for suggesting the example at the end of Section~\ref{TOW}.

\small
\bibliographystyle{plain}
\bibliography{ah}

\end{document}